\documentclass[12pt]{amsart}
\usepackage{fullpage}
\usepackage{mathrsfs}
\usepackage{hyperref}
\usepackage{cite}
\usepackage[utf8]{inputenc}
\usepackage{mathtools,amsmath,amssymb}
\usepackage{enumitem}
\usepackage{physics}

\newtheorem{thm}{Theorem}[section]
\newtheorem{lem}{Lemma}[section]
\newtheorem{cor}{Corollary}[section]
\newtheorem{prop}{Proposition}[section]

\theoremstyle{remark}
\newtheorem{remark}{Remark}[section]


\title[Higher Mertens constants]{Higher Mertens constants for\\ almost primes}

\author[J.\ Bayless]{Jonathan Bayless}

\address{University of Maine at Augusta, 46 University Drive, Augusta, Maine 04330}

\email{jonathan.bayless@maine.edu}

\author[P.\ Kinlaw]{Paul Kinlaw}

\address{Husson University, 1 College Circle
Bangor, Maine 04401}

\email{kinlawp@husson.edu}

\author[J.\ D.\ Lichtman]{Jared Duker Lichtman}

\address{Mathematical Institute, University of Oxford, Oxford, OX2 6GG, UK}

\email{jared.d.lichtman@gmail.com}

\date{January 27, 2022}

\begin{document}

\begin{abstract}
For $k\ge1$, a $k$-almost prime is a positive integer with exactly $k$ prime factors, counted with multiplicity. In this article we give elementary proofs of precise asymptotics for the reciprocal sum of $k$-almost primes. Our results match the strength of those of classical analytic methods. We also study the limiting behavior of the constants appearing in these estimates, which may be viewed as higher analogues of the Mertens constant $\beta=0.2614..$. Further, in the case $k=2$ of semiprimes we give yet finer-scale and explicit estimates, as well as a conjecture.
\end{abstract}

\maketitle

\section{Introduction}
For a positive integer $n$ denote by $\Omega(n)$ the number of prime factors of $n$, counted with multiplicity. For $k\ge1$, a $k$-almost prime is a positive integer $n$ with $\Omega(n)=k$.

In this paper we consider the reciprocal sum of $k$-almost primes,
\begin{align}\label{eq:Rkx}
{\mathcal R}_k(x) = \mathop{\sum_{\Omega(n)=k}}_{n\le x}\frac{1}{n} = \mathop{\sum_{p_1\cdots p_k\le x}}_{p_1\le \ldots\le p_k}\frac{1}{p_1\cdots p_k}.
\end{align}
In estimating this sum, we generalize Mertens' second theorem \cite{mertens},
\begin{align*}
\sum_{p\le x}\frac{1}{p} = \log_2 x + \beta + O\left(\frac{1}{\log x}\right),
\end{align*}
where $\beta=0.2614\ldots$ is the Mertens (or Meissel--Mertens) constant, which satisfies
\begin{equation}\label{mertensconstant}
\beta=\gamma+\sum_p\left(\frac{1}{p}+\log\left(1-\frac{1}{p}\right)\right).
\end{equation}
Here $\gamma=0.5772\ldots$ is the Euler--Mascheroni constant. 

Throughout the paper, we denote $p$ as a prime variable and $\log_2 x=\log(\log x)$ as the iterated natural logarithm. In addition, $\Gamma$ is the Euler gamma function and $\zeta$ is the Riemann zeta function.

Our main result is the following.

\begin{thm}\label{mainthm2}
For $x\ge3$, we have
\begin{equation}\label{Vk}
{\mathcal R}_k(x) = \sum_{j=0}^k\frac{\nu_{k-j}}{j!}(\log_2 x)^j + O_k\left(\frac{(\log_2 x)^{k-1}}{\log x}\right),
\end{equation}
where $\nu_j := \nu^{(j)}(0)/j!$ are the Taylor coefficients of
\begin{equation}\label{nu}
\nu(z) = \frac{1}{\Gamma(z+1)}\prod_{p}\left(1-\frac{z}{p}\right)^{-1}\left(1-\frac{1}{p}\right)^z.
\end{equation}
\end{thm}
The proof of Theorem \ref{mainthm2} is completely elementary, see Section \ref{sec:perspect} for further remarks.

For comparision, classical analytic estimates imply the slightly cruder result
\begin{equation}\label{crude}
{\mathcal R}_k(x) = \sum_{j=1}^k \frac{\nu_{k-j}}{j!} (\log_2 x)^j + C_k + O_k\left(\frac{(\log_2 x)^{k}}{\log x}\right)
\end{equation}
for some (a priori unspecified) constant $C_k$.\footnote{The exponent $k$ in the error term of \eqref{crude} corrects the stated exponent $k-1$ in the published version.} Indeed, the Sathe-Selberg theorem \cite[Theorem II.6.5]{tenenbaumintro} states that, for each $0<\delta<1$, the counting function $\mathcal N_k(x)$ of $k$-almost primes is
\begin{equation}\label{selberg}
\mathcal N_k(x)=\frac{x}{\log x}\sum_{j=0}^{k-1} \frac{\nu_{k-1-j}}{j!} (\log_2 x)^j+O_\delta\left(\frac{x(\log_2 x)^k}{k!(\log x)^2}\right)
\end{equation}
uniformly for $x\ge 3$ and $k\le (2-\delta)\log_2 x$. Then by partial summation, \eqref{selberg} implies \eqref{crude} with constant coefficient
\begin{equation}\label{const}
C_k=\int_{2^k}^\infty \frac{\mathcal{E}_k(t)}{t^2}\dd{t} -  \sum_{j=1}^k \frac{\nu_{k-j}}{j!}(\log_2 2^k)^j,
\end{equation}
where $\mathcal{E}_k(t)$ is the error term in \eqref{selberg}. Similarly as with $\nu_1=\beta$ in the original proof of Mertens' second theorem, the added value of Theorem \ref{mainthm2} lies in the identification of $C_k = \nu_k$.

Moreover, though not immediately apparent from the statement itself, our proof of Theorem \ref{mainthm2} directly implies the following concise bounds, which may be of independent interest.

\begin{cor}\label{Rkbounds}
For each $k\ge 2$ there exists $x_0(k)>0$ such that
\begin{displaymath}
\frac{1}{k!}\,(\log_2 x)^k \ < \ {\mathcal R}_k(x)  \ < \ \frac{1}{k!}\,(\log_2 x+\beta)^k
\end{displaymath}
for all $x > x_0(k)$, where $\beta$ is the Mertens constant.
\end{cor}

The lower bound of Corollary \ref{Rkbounds} holds when $k=1$ by Mertens' second theorem, in fact for all $x>1$ (see for instance \cite{rosser1962approximate}), while the upper bound fails to hold when $k=1$.  By a result of G.\ Robin \cite{robin}, the error term in Mertens' second theorem changes sign infinitely often.
As a consequence of Theorem \ref{explicitr2theorem} below, we find that the optimal value of $x_0(2)$ is $\exp(\exp(\sqrt{113/90}-\beta)) = 10.5998\ldots$. (See Corollary \ref{r2bounds} below.)

We note that Theorem \ref{mainthm2} generalizes Mertens' second theorem, which is the special case $k=1$. In particular, the constant coefficients $\nu_k$ may be viewed as higher analogues of the Mertens constant $\nu_1=\beta$.

We establish a recurrence for $\nu_k$ (see Proposition \ref{prop:recurrence}) which enables rapid computation of $\nu_k$ to high precision. We display results for $k\le 20$ below.
\[\begin{array}{rrrr}
k & \nu_k & k & \nu_k\\
\hline
1 & \beta=2.61497\cdot10^{-1} & 11 & 1.49365\cdot10^{-4}\\
2 & -5.62153\cdot10^{-1}   & 12 & 4.99174\cdot 10^{-5}\\
3 & 3.05978\cdot10^{-1}    & 13 & 1.82657\cdot 10^{-5}\\
4 & 2.62973\cdot10^{-2}    & 14 & 1.30241\cdot 10^{-5}\\
5 & -6.44501\cdot10^{-2}   & 15 & 5.52779\cdot 10^{-6}\\
6 & 3.64064\cdot10^{-2}    & 16 & 2.90194\cdot 10^{-6}\\
7 & -4.70865\cdot10^{-3}   & 17 & 1.45075\cdot 10^{-6}\\
8 & -4.33984\cdot10^{-4}   & 18 & 7.19861\cdot 10^{-7}\\
9 & 1.5085\cdot 10^{-3}    & 19 & 3.61606\cdot 10^{-7}\\
10 & -1.83548\cdot10^{-4}  & 20 & 1.80517\cdot 10^{-7}\\    
\end{array}\]

Though initially appearing rather erratic, the table eventually suggests $\nu_k \approx 0.189\cdot 2^{-k}$. 

In light of (or perhaps despite) this numerical evidence, we show the higher Mertens constants $\nu_k$ satisfy the following precise asymptotics.

\begin{thm}\label{nuexpand}
We have $\nu_k = \beta_2 \,2^{-k}+O(3^{-k})$, where $\beta_2 = \frac{1}{8}\prod_{p>2}\frac{(1-1/p)^2}{1-2/p} \approx 0.1893475$.

Moreover for any prime $q$, we have
\begin{align}\label{eq:expand}
\nu_k = \sum_{p<q}\beta_p \,p^{-k} \ + \ O_q(q^{-k}),
\end{align}
where $\beta_2,\beta_3,\beta_5,...$ are constants given by
\begin{align}\label{betap}
\beta_p := \frac{(1-1/p)^p}{p!}\prod_{p'\neq p}\Big(1-\frac{p}{p'}\Big)^{-1}\Big(1-\frac{1}{p'}\Big)^{p}.
\end{align}
\end{thm}
The proof of Theorem \ref{nuexpand} is in Section \ref{recurrence}, also by elementary combinatorial methods. For an outline of the arguments see Section \ref{sec:perspect}.

\subsection{Finer-scale estimates for semiprimes.}
Let $k=2$. We establish a finer estimate for ${\mathcal R}_2(x)$.  Let $E(x) = \sum_{p\le x}1/p -(\log_2 x+\beta)$ denote the error term in Mertens' second theorem.

\begin{thm}\label{r2theorem}
For any $N\ge 0$ we have
\begin{displaymath}
{\mathcal R}_2(x)=\frac{1}{2}(\log_2 x+\beta)^2+\frac{P(2)-\zeta(2)}{2}+\sum_{1\le j\le N}\frac{\alpha_j}{\log^j x}+O_N\left((\log x)^{-N-1}\right),
\end{displaymath}
where $\alpha_j$ are constants given by
\begin{align}\label{eq:alphaj}
\alpha_j&:=\frac{\log^j 2}{j}\left(\frac{1}{j}-\log_2 2-\beta\right)+\int_2^\infty \frac{E(t)\log^{j-1} t}{t}\dd{t}.
\end{align}
\end{thm}

As we shall prove in Lemma \ref{lem:alphaj}, the constants $\alpha_j$ can be viewed as higher analogues of the constant
\begin{displaymath}
\alpha_1 = \lim_{x\to \infty}\left(\log x - \sum_{p\le x}\frac{\log p}{p}\right) = \gamma + \sum_{m\ge 2}\sum_p \frac{\log p}{p^m} = 1.332582\ldots
\end{displaymath}
arising from a strong form of Mertens' first theorem.  See for instance \cite[Theorem 5.7]{dusart2016}.  We also determine an explicit bound for ${\mathcal R}_2(x)$.

\begin{thm}\label{explicitr2theorem}
For all $x > 1$ we have
\begin{displaymath}
\left|{\mathcal R}_2(x) \ - \ \frac{1}{2}(\log_2 x+\beta)^2-\frac{P(2)-\zeta(2)}{2}-\frac{\alpha_1}{\log x}\right| \ < \ (\log x)^{-3/2}.
\end{displaymath}
\end{thm}

As a consequence we have the following bounds which make Corollary \ref{Rkbounds} explicit for the case $k=2$.

\begin{cor}\label{r2bounds}
We have
\begin{displaymath}
\frac{1}{2}\,(\log_2 x)^2 \ < \ {\mathcal R}_2(x)  \ < \ \frac{1}{2}\,(\log_2 x+\beta)^2,
\end{displaymath}
where the lower bound holds for all $x\ge 4$ and the upper bound holds for all $x\ge 11$.
\end{cor}

We will prove Theorems \ref{r2theorem} and \ref{explicitr2theorem} in Section \ref{R2}. We also consider the sum ${\mathcal R}_2^*(x)$ restricted to squarefree semiprimes (i.e. products of 2 distinct primes).

\begin{cor}\label{r2squarefree}
For any $N\ge 0$, ${\mathcal R}_2^*(x)$ satisfies the estimate in Theorem \ref{r2theorem} with $P(2)/2$ replaced by $-P(2)/2$.  Moreover, we have
\begin{displaymath}
\left|{\mathcal R}_2^*(x) \ - \ \frac{1}{2}(\log_2 x+\beta)^2+\frac{P(2)+\zeta(2)}{2}-\frac{\alpha_1}{\log x}\right| \ < \ (\log x)^{-3/2},
\end{displaymath}
where the upper bound on ${\mathcal R}_2^*(x)$ holds for all $x\ge 227$ and the lower bound holds for all $x>1$.
\end{cor}

The corollary follows from the bounds established in the proof of Theorem \ref{explicitr2theorem} together with Lemma \ref{pn} below.

We learned that constants related to $\alpha_j$ from \eqref{eq:alphaj} appear in a very recent preprint \cite{crisanerban} of Crisan and Erban who derive an asymptotic expansion for the counting function $\mathcal N_2(x)$ of semiprimes and compute the values of $B_j:=-j\alpha_j$ to high precision.  By extending their table of $B_j$ for $1\le j\le 10$, we are led to conjecture that
\begin{align}\label{eq:alphaj1j}
\frac{\alpha_{j+1}}{\alpha_j} \ \sim \ \frac{2 j^2}{j+1}.
\end{align}
Indeed, equation (3.5) in \cite{crisanerban} yields the following numerical values for $\alpha_j$ and $(\alpha_{j+1}/\alpha_j)/(2j^2/(j+1))$ for $11\le j\le 20$, computed in Pari/GP.

\[\begin{array}{ccc}
j & \alpha_j & (\alpha_{j+1}/\alpha_j)/(2j^2/(j+1))\\
\hline
11 & 3.4791\cdot10^{8} & 0.98998\\
12 & 6.9638\cdot10^{9} & 0.99253\\
13 & 1.5342\cdot10^{11} & 0.99443\\
14 & 3.6886\cdot10^{12}  & 0.99585\\
15 & 9.6096\cdot10^{13} & 0.99691\\
16 & 2.6965\cdot10^{15}   & 0.99769\\
17 & 8.1071\cdot10^{16} & 0.99828\\
18 & 2.5999\cdot10^{18}     & 0.99871\\
19 & 8.8587\cdot 10^{19} & 0.99904\\
20 & 3.1957\cdot10^{21}    & 0.99928\\    
\end{array}\]

\begin{remark}
In private communication, Ofir Gorodetsky has given a sketch proof that
\begin{align*}
\alpha_j \sim (j-1)!\frac{2^{j-1}}{j},
\end{align*}
which immediately implies \eqref{eq:alphaj1j}. The basic approach is to use the representation for $\alpha_j$ obtained in Lemma \ref{lem:alphaj}, and apply the von Mangoldt explicit formula (with some care). Full details will be provided in forthcoming work.
\end{remark}

We note that the work of Crisan and Erban \cite{crisanerban} may suggest an alternate approach to Theorem \ref{r2theorem}.  Conversely, Theorem \ref{r2theorem} can be used to give a new proof of Crisan and Erban's asymptotic expansion of $\mathcal{N}_2(x)$.

\begin{thm}[Theorem 2.5 \cite{crisanerban}]\label{N2}
For any $N\ge 0$ we have 
\begin{displaymath}
\mathcal{N}_2(x) = \frac{x}{\log x}\sum_{n=0}^{N-1} n!\,\frac{\log_2 x + D_n}{\log^n x} + O_N\left(\frac{x\log_2 x}{\log^{N+1} x}\right),
\end{displaymath}
for constants $D_n = \sum_{j=0}^n B_j/j! - H_n$, where $H_n$ is the $n$-th harmonic number, $B_0=\beta$, and $B_j=-j\alpha_j,~j\ge 1$, for $\alpha_j$ as in \eqref{eq:alphaj}.
\end{thm}

The new proof of Theorem \ref{N2} is provided in Section \ref{sec:N2}.  It relies on a strong form of the prime number theorem, which implies the estimate
\begin{displaymath}
E(x)\ll_A (\log x)^{-A}~\text{for all}~A>0.
\end{displaymath}
We will make use of this estimate throughout the paper.  We remark that it follows from a version of the prime number theorem established by elementary methods.  See for instance \cite[Sec.\ 8]{diamond}.

\section{Outline and perspectives on the argument}\label{sec:perspect}

Our proof for ${\mathcal R}_k(x)$ proceeds by building on recent work of Tenenbaum \cite{tenenbaum} and Qi--Hu \cite{qihu} for the related sum
\begin{align}
{\mathcal S}_k(x) := \mathop{\sum_{p_1\cdots p_k\le x}}\frac{1}{p_1\cdots p_k} = \sum_{\substack{\Omega(n)=k\\ n\le x}}\frac{{\rm f}(n)}{n},
\end{align}
where ${\rm f}(n)$ denotes the number of ordered prime factorizations of $n$. Recall \eqref{eq:Rkx} for comparision.

\begin{thm}[Theorem 1 \cite{tenenbaum}]\label{tenenbaumthm}
For $x\ge3$,
\begin{equation}\label{tenenbaumsresult}
{\mathcal S}_k(x) = S_k(\log_2 x) + O_k\left(\frac{(\log_2 x)^{k-1}}{\log x}\right),
\end{equation}
where $S_k(X)=\sum_{0\le j\le k}\lambda_{j,k}X^j$ is a monic polynomial of degree $k$ and
\begin{equation}\label{lambda}
\lambda_{j,k} = \sum_{m=0}^{k-j}\binom{k}{m,j,k-m-j} (\beta-\gamma)^{k-m-j}\left(\frac{1}{\Gamma}\right)^{(m)}(1).
\end{equation}
\end{thm}

Tenenbaum's analytic proof employs the Selberg--Delange method. However, Qi and Hu \cite{qihu} have recently obtained an elementary proof via Dirichlet's hyperbola method.

\begin{thm}[Theorem 1.3 \cite{qihu}]\label{qihuthm}
For $x\ge3$,
\begin{equation}\label{qihuresult}
{\mathcal S}_k(x) = \sum_{j=0}^k \binom{k}{j} a_{k-j}(\log_2 x+\beta)^j + O_k\left(\frac{(\log_2 x)^{k-1}}{\log x}\right),
\end{equation}
where the coefficients $a_j$ are defined recursively by $a_0=1, a_1=0$, and 
\begin{align*}
a_j = \sum_{i=1}^{j-1}(-1)^i \binom{j-1}{i} i! \zeta(i+1)a_{j-1-i}.
\end{align*}
\end{thm}

The results of Tenenbaum and Qi--Hu generalize those of Popa \cite{popadouble,popatriple} in the cases $k=2$ and $k=3$ (as well as Mertens when $k=1$). Though it is not immediately clear that the expansions in Theorems \ref{tenenbaumthm} and \ref{qihuthm} are equivalent, we will show this in Section \ref{qihutenenb}.

In the sum ${\mathcal S}_k(x)$, terms can appear up to $k!$ times due to permutation of the prime factors. We obtain the following estimate for ${\mathcal R}_k(x)$, which counts each term precisely once.  

\begin{thm}\label{mainthm}
For $x\ge3$, we have
\begin{displaymath}
{\mathcal R}_k(x) = R_k(\log_2 x+\beta)+O_k\left(\frac{(\log_2 x)^{k-1}}{\log x}\right),
\end{displaymath}
where $R_k$ is the polynomial of degree $k$ defined by
\begin{displaymath}
R_k(X) = \sum_{n_1+2n_2+\ldots+kn_k=k}\sum_{i=0}^{n_1} \frac{a_{i,{n_1}}}{n_1!}\prod_{j=2}^k\left(\frac{P(j)}{j}\right)^{n_j}\frac{X^j}{n_j!},
\end{displaymath}
where the sum ranges over all integer partitions of $k$, $P(j)=\sum_p p^{-j}$ is the prime zeta function, and the coefficients $a_{i,n}$ are given by
\begin{displaymath}
a_{i,n}=\sum_{m=0}^{n-i}{{n}\choose{m,i,n-m-i}}(-\gamma)^{n-m-i}\left(\frac{1}{\Gamma}\right)^{(m)}(1).
\end{displaymath}
\end{thm}

When combined with the Weierstrass product formula (see \eqref{WeierstrassGamma}), Theorem \ref{mainthm} yields the following special cases:
\begin{align*}
{\mathcal R}_2(x) & = \frac{1}{2}\left(\log_2 x+\beta\right)^2 +  \frac{P(2)-\zeta(2)}{2} + E_2(x),\\
{\mathcal R}_3(x) & = \frac{1}{6}(\log_2 x+\beta)^3+\frac{P(2)-\zeta(2)}{2}(\log_2 x+\beta)+\frac{P(3)+\zeta(3)}{3}+E_3(x),\\
{\mathcal R}_4(x) &= \frac{1}{24}(\log_2 x+\beta)^4+\frac{P(2)-\zeta(2)}{4}(\log_2 x+\beta)^2 +E_4(x)\\
&\quad +\frac{P(3)+\zeta(3)}{3}(\log_2 x+\beta)+\frac{P(4)}{4}+\frac{\zeta(4)}{16} +\frac{P(2)^2}{8}-\frac{P(2)\zeta(2)}{4}.
\end{align*}
Here $\zeta$ is the Riemann zeta function and
\begin{displaymath}
E_k(x):={\mathcal R}_k(x) - R_k(\log_2 x+\beta)\ll_k (\log_2 x)^{k-1}/\log x.
\end{displaymath}

We observe that since $P(2)-\zeta(2)<0$, Theorem \ref{mainthm} implies Corollary \ref{Rkbounds}.

Recall Theorem \ref{mainthm2} states ${\mathcal R}_k(x) = V_k(\log_2 x) + O_k((\log_2 x)^{k-1}/\log x)$ for the polynomial
\begin{equation}\label{V}
V_k(X) := \sum_{j=0}^k \frac{\nu_{k-j}}{j!}X^j.
\end{equation}

In Sections \ref{main}-\ref{main2}, we shall establish Theorem \ref{mainthm} from Tenenbaum's result by an elementary combinatorial argument. By a similar argument we prove the polynomial identity
\begin{align}
R_k(X+\beta)=V_k(X)
\end{align}
which gives the equivalence of Theorems \ref{mainthm} and \ref{mainthm2}. This allows us to obtain Theorem \ref{mainthm2}, thus implying that $C_k=\nu_k$. Hence, when combined with the hyperbola method in Theorem \ref{qihuthm}, the proof of Theorems \ref{mainthm2} and \ref{mainthm} are completely elementary in nature.

As the Sathe--Selberg theorem \eqref{selberg} is proven by contour integration, Theorem \ref{mainthm2} represents a success of elementary methods at matching the strength of results with those of more sophisticated analytic tools. One may wonder if similar elementary methods could be used to establish the Sathe--Selberg theorem itself.

\subsection{Outline of Theorem \ref{nuexpand}}

We note that the proof methods in Theorem \ref{mainthm} to handle $\mathcal R_k(x)$ are motivated by those of \cite[Theorem 2]{lichtman2} for the smooth variant
\begin{align}
\mathcal P_k(x) :=\sum_{\substack{\Omega(n)=k\\P^+(n)\le x}}\frac{1}{n},
\end{align}
where $P^+(n)$ is the largest prime factor of $n$. As noted, the proof of Theorem \ref{mainthm} proceeds by passing from $\mathcal R_k(x)$ to $\mathcal S_k(x)$. Whereas, the analogous method for $\mathcal P_k(x)$ may be applied directly, and is quite natural in the smooth setting.

Theorems 1.2 and 3.5, as well as (2.3), in \cite{lichtman2} give the following.

\begin{thm}[\cite{lichtman2}]\label{thm:smooth}
For each fixed $k\ge1$, we have
\begin{align*}
\mathcal P_k(x) \ = \ \sum_{j=0}^k\frac{d_{k-j}}{j!}\big(\log_2 x + \beta\big)^j \ & + \ O_k\Big(\frac{(\log_2 x)^{k-1}}{\log x}\Big),
\end{align*}
where the sequence $(d_k)_{k=0}^\infty$ is recursively defined by $d_0=1$ and $d_k = \frac{1}{k}\sum_{j=2}^k d_{k-j}\,P(j)$. Explicitly, we have
\begin{align*}
d_k = \sum_{2n_2+3n_3+\cdots =k}\prod_{j\ge2}\frac{(P(j)/j)^{n_j}}{n_j!},
\end{align*}
where the sum ranges over partitions of $k$ without singletons. 

Moreover, for each prime $q$,
\begin{equation}\label{dexpand}
d_k = \sum_{p<q}\delta_p\,p^{-k} \ + \ O_q(q^{-k}),
\end{equation}
where $\delta_p:=e^{-1}\prod_{q\neq p}(1-\tfrac{p}{q})^{-1}e^{-p/q}$. In particular $d_k = \delta_2\,2^{-k} + O(3^{-k})$.
\end{thm}

In Section \ref{recurrence}, we utilize Theorem \ref{thm:smooth} in order to establish Theorem \ref{nuexpand}.

Finally, in Theorem \ref{nuexpand}, it is straightforward to verify that if such an expansion \eqref{eq:expand} exists, then the coefficients are determined uniquely by \eqref{betap}.  Indeed, this is seen by comparing the residue from the poles of $\nu(z)$ at $z=p$ with that of
\begin{align*}
\nu(z) & = \sum_{k\ge0} \nu_k z^k \ = \  \sum_{p<q}\beta_p \sum_{k\ge0} (z/p)^{k} \ + \ O_q(\sum_{k\ge0} (z/q)^{k})\\
& = \sum_{p<q}\frac{\beta_p}{1-z/p} \ + \ \frac{O_q(1)}{1-z/q}.
\end{align*}
Hence the core of the proof in Theorem \ref{nuexpand} lies in the existence of the expansion \eqref{eq:expand}.

This situation is somewhat reminiscent of the developments leading to the prime number theorem (PNT), i.e. $\pi(x) \sim x/\log x$. Indeed, Mertens' second theorem did not directly imply PNT. However, if the expansion $\pi(x) \sim cx/\log x$ were assumed to exist for some constant $c>0$, then Mertens' theorem would force $c=1$ (this result is historically attributed to Chebyshev). As such the core of the proof of PNT was to establish existence of such an expansion.

\section{The Proof of Theorem \ref{mainthm}}\label{main}

To prove Theorem \ref{mainthm} it suffices to establish the following result, which allows us to apply Tenenbaum's estimate \eqref{tenenbaumsresult}.  

\begin{prop}\label{prop1}
For $x\ge 3$ we have
\begin{displaymath}
{\mathcal R}_k(x) = \sum_{n_1+2n_2+\ldots+kn_k=k}\frac{{\mathcal S}_{n_1}(x)}{n_1!}\cdot\prod_{j=2}^k\frac{(P(j)/j)^{n_j}}{n_j!}+O_k\left(\frac{(\log_2 x)^{k-1}}{\log x}\right).
\end{displaymath}
\end{prop}

This result may be viewed as an analogue of \cite[Proposition 3.1]{lichtman} and \cite[Proposition 2.1]{lichtman2}. These latter results concern ``smooth'' (or ``friable'') sums over $P^+(n)\le x$, in which setting admits concise combinatorial proofs. However, the technical difficulty for Proposition \ref{prop1} arises, since the usual sums over $n\le x$ in ${\mathcal R}_k(x)$ require greater care in order to facilitate the combinatorics.

We will establish Proposition \ref{prop1} after stating and proving Propositions \ref{inclusionexclusion} and \ref{exponenttype} below.

\begin{prop}\label{inclusionexclusion}
We have
\begin{displaymath}
{\mathcal R}_k(x) = \sum_{n_1+2n_2+\ldots+kn_k=k}\left(\prod_{j=1}^k\frac{j^{-n_j}}{n_j!}\right)\sum_{m\le x}\frac{f(m)}{m},
\end{displaymath}
where $f(m)=f_{(n_1,...,n_k)}(m) = \#\{ \text{ordered prime tuples }(p_{1j},..,p_{n_jj})_{j\le k} : m=\prod_{j=1}^k \prod_{i=1}^{n_j}p_{ij}^j\}$.
\end{prop}

\begin{proof}
We apply Perron's formula to the characteristic function of the $k$-almost primes, along with the inclusion-exclusion formulas for the almost prime zeta function $P_k(z)=\sum_{\Omega(n)=k}n^{-z}$ (in contrast to $P(z)^k$ as in Tenenbaum's proof \cite{tenenbaum}) to obtain for any $c>0$ and $x\in \mathbb{R}^+\setminus \mathbb{N}$,
\begin{align*}
{\mathcal R}_k(x) &= \sum_{\substack{\Omega(n)=k\\n\le x}}\frac{1}{n} \ = \ \frac{1}{2\pi i}\int_{c+i\mathbb{R}}P_k(z+1)\frac{x^z}{z}\dd{z}\\
&= \frac{1}{2\pi i}\int_{c+i\mathbb{R}}\sum_{n_1+2n_2+\ldots+kn_k=k}\prod_{j=1}^k\frac{1}{n_j!}\left(\frac{P(j(z+1))}{j}\right)^{n_j}\frac{x^z}{z}\dd{z}\\
&= \sum_{n_1+2n_2+\ldots+kn_k=k}\frac{1}{2\pi i}\left(\prod_{j=1}^k\frac{j^{-n_j}}{n_j!}\right)\int_{c+i\mathbb{R}}\prod_{j=1}^k P(j(z+1))^{n_j}\frac{x^z}{z}\dd{z},
\end{align*}
where the equality of the first and second lines is a direct application of \cite[Proposition 3.1]{lichtman}.  Applying Perron's formula again,
\begin{align}\label{eq:Rkf}
{\mathcal R}_k(x) &= \sum_{n_1+2n_2+\ldots+kn_k=k}\left(\prod_{j=1}^k\frac{j^{-n_j}}{n_j!}\right)\sum_{m\le x}\frac{f(m)}{m},
\end{align}
where $f$ satisfies
\begin{align*}
\sum_{m=1}^\infty \frac{f(m)}{m^s} = \prod_{j=1}^k P(js)^{n_j} = \prod_{j=1}^k \prod_{i=1}^{n_j}\sum_{p_{ij}} p_{ij}^{-js} = \sum_{{A}} \prod_{q\in A}q^{-s}
\end{align*}
as $A$ ranges over all choices of multi-sets $A = \{p_{ij}^j : j\le k, i\le n_j\}$.  Thus $f(m) = \sum_{m = \prod_{j=1}^k \prod_{i=1}^{n_j}p_{ij}^j}1$, by the uniqueness of Dirichlet series coefficients. Hence
\begin{align}\label{eq:fnx}
\sum_{m\le x}\frac{f(m)}{m} = \sum_{ \prod_{j=1}^k \prod_{i=1}^{n_j}p_{ij}^{j} \, \le \, x}\prod_{j=1}^k \prod_{i=1}^{n_j}p_{ij}^{-j}.
\end{align}
Substituting this back into \eqref{eq:Rkf} completes the proof.
\end{proof}

We note that Proposition \ref{inclusionexclusion} is completely elementary and combinatorial in nature. However, it is efficiently deduced with the help of Perron's formula.

To complete the proof of Proposition \ref{prop1} we show that (at an admissible error) we may replace the condition $\prod_{j=1}^k (p_{1j}\cdots p_{n_j j})^j \, \le \, x$ in the proof of Proposition \ref{inclusionexclusion} with the weaker condition $p_{11}\cdots p_{n_1 1} \, \le \, x$, to obtain
\begin{displaymath}
\sum_{m\le x}\frac{f(m)}{m} \approx \Big(\sum_{p_{11}\cdots p_{n_11}\le x}\prod_{i=1}^{n_1}p_{i1}^{-1}\Big)\prod_{j=2}^k \prod_{i=1}^{n_j} \sum_{p_{ij}} p_{ij}^{-j} \ = {\mathcal S}_{n_1}(x)\prod_{j=2}^k P(j)^{n_j}.
\end{displaymath}

We shall do so using the following proposition, whose proof we give in the next section.

\begin{prop}\label{exponenttype}
Given $\ell\ge 0, n\ge 0,$ and $a_i\ge 2$ for $i=1,\ldots,\ell$, let $k=n+\sum_{i\le \ell} a_i$. Then
\begin{displaymath}
\sum_{m=q_1\cdots q_n \prod_{i\le \ell}p_i^{a_i}\le x}\frac{1}{m} = {\mathcal S}_n(x)\prod_{i\le \ell} P(a_i)+O_{k}\left(\frac{(\log_2 x)^{n-1}}{\log x}\right),
\end{displaymath}
where the sum $\sum_{m\le x}$ ranges over $n+\ell$ independent prime variables, $q_1,\ldots q_n,p_1\ldots p_\ell$, whose values are not necessarily distinct.
\end{prop}

\begin{proof}[Proof of Proposition \ref{prop1} from Propositions \ref{inclusionexclusion} and \ref{exponenttype}]

From \eqref{eq:fnx},
\begin{align}
\sum_{m\le x}\frac{f(m)}{m} &= \sum_{m=p_{11}\cdots p_{n_1 1}\prod_{j=2}^k\prod_{i=1}^{n_j} p_{ij}^j\le x}\frac{1}{m} \nonumber\\
& = {\mathcal S}_{n_1}(x)\prod_{j=2}^k P(j)^{n_j}+O_k\left(\frac{(\log_2 x)^{n_1-1}}{\log x}\right) \label{eq:fmmSn}
\end{align}
by Proposition \ref{exponenttype} for the choices of parameters
\begin{align*}
n=n_1,\quad \ell = 2n_2+\cdots+kn_k, \quad \{a_1,\ldots,a_\ell\} = \bigcup_{j=2}^k\{j\}^{n_j} \\
q_1,\ldots q_n = p_{11},\ldots p_{n_1 1}, \quad \prod_{i=1}^{\ell}p_i^{a_i} = \prod_{j=2}^k\prod_{i=1}^{n_j} p_{ij}^j.
\end{align*}
Substituting \eqref{eq:fmmSn} into Proposition \ref{inclusionexclusion} gives
\begin{align*}
{\mathcal R}_k(x) & = \sum_{n_1+2n_2+\ldots+kn_k=k}\left(\prod_{j=1}^k\frac{j^{-n_j}}{n_j!}\right)\bigg({\mathcal S}_{n_1}(x)\prod_{j=2}^k P(j)^{n_j}+O_k\left(\frac{(\log_2 x)^{n_1-1}}{\log x}\right)\bigg)\\
& = \sum_{n_1+2n_2+\ldots+kn_k=k}\frac{{\mathcal S}_{n_1}(x)}{n_1!}\cdot\prod_{j=2}^k\frac{(P(j)/j)^{n_j}}{n_j!}+O_k\left(\frac{(\log_2 x)^{k-1}}{\log x}\right).
\end{align*}
This completes the proof of Proposition \ref{prop1}, and therefore Theorem \ref{mainthm}.
\end{proof}

\section{Proof of Proposition \ref{exponenttype}}

We will use a lemma of Nguyen and Pomerance, see \cite[Lemma 2.7]{pomeranceamicable}.

\begin{lem}[Nguyen and Pomerance]\label{pn}
For all $x>1$, we have
\begin{displaymath}
\sum_{p>x}\frac{1}{p^2} < \frac{1}{x\log x}.
\end{displaymath}
\end{lem}

It follows that $\sum_{p>x}p^{-a} < (x^{a-1}\log x)^{-1}$ for any $x>1$ and $a\ge 2$.  Aside from the numerically explicit bounds in Theorem \ref{explicitr2theorem} and Corollaries \ref{r2bounds}--\ref{r2squarefree}, we only use an upper bound of the form $O((x^{a-1}\log x)^{-1})$, which is a consequence of Chebyshev's estimates.

\begin{proof}[Proof of Proposition \ref{exponenttype}]

We proceed by induction on $\ell$.  Note the claim holds for $\ell=0$ by definition of ${\mathcal S}_n(x)$. The claim also holds for $\ell=1$ and $n=0$ by Lemma \ref{pn}, with the convention that ${\mathcal S}_0(x)=1$. So assume $\ell=1$, and take any $n\ge 1$, and $a\ge 2$. 

Let $\mathcal{E}$ denote the multi-set of numbers of the form $q_1\cdots q_n$. Then
\begin{align}\label{eq:pam}
\mathop{\sum_{p^a m\le x}}_{m\in \mathcal{E}} \frac{1}{p^a m} &= \sum_{p\le (x/2^n)^{1/a}}\frac{1}{p^a}\mathop{\sum_{m\le x/p^a}}_{m\in \mathcal{E}}\frac{1}{m} = \sum_{p\le (x/2^n)^{1/a}}\frac{1}{p^a}{\mathcal S}_n(x/p^a) \nonumber\\
&= \sum_{p\le (x/2^n)^{1/a}}\left(\frac{1}{p^a}S_n\left(\log_2(x/p^a)\right)+O_{n}\left(\frac{1}{p^a}g_n(x/p^a)\right)\right) 
\end{align}
by \eqref{tenenbaumsresult} from Tenenbaum, where $g_n(x):=(\log_2 x)^{n-1}/\log x$.  (Note the condition $t\ge 3$ in \eqref{tenenbaumsresult}.  Indeed $x/p^a\ge 3$ if $n\ge 2$, while if $n=1$, we appeal directly to Mertens' theorem.) Splitting the sum at ${x}^{1/(a+1)}$, we bound the error in \eqref{eq:pam} as
\begin{align}\label{eq:pagn}
\sum_{p\le (x/2^n)^{1/a}}\frac{g_n(x/p^a)}{p^a}
\le (\log_2 x)^{n-1}\Big(\frac{a+1}{a\log x}\sum_{p\le x^{1/(a+1)}}p^{-a}  + \frac{1}{n\log 2}\sum_{x^{1/(a+1)}< p}p^{-a}\Big) \ll_{a,n} g_n(x)
\end{align}
using Lemma \ref{pn} on the right sum above.

Let $s_a(t) :=\sum_{p\le t}p^{-a}$. By partial summation, the main term in \eqref{eq:pam} is
\begin{align*}
\sum_{p\le (x/2^n)^{1/a}}\frac{1}{p^a}S_n\left(\log_2(x/p^a)\right) & = s_a\left((x/2^n)^{1/a}\right)S_n(\log_2 2^n)+\int_2^{(x/2^n)^{1/a}}\frac{as_a(t)S_n'\left(\log_2(x/t^a)\right)}{t\log(x/t^a)}\dd{t}\\
& = P(a)S_n(\log_2 2^n) + P(a)\int_2^{(x/2^n)^{1/a}}\frac{aS_n'\left(\log_2(x/t^a)\right)}{t\log(x/t^a)}\dd{t}\\
& \quad+O_{a,n}\left(x^{1/a-1}+\int_2^{(x/2^n)^{1/a}}\frac{a\left|S_n'\left(\log_2(x/t^a)\right)\right|}{t^a\log t\log(x/t^a)}\dd{t}\right)
\end{align*}
using Lemma \ref{pn}. We bound the latter integral in the error as $\ll_{a,n}g_n(x)$ by splitting the interval at $y=x^{1/(a+1)}$.  (If $t>y$ then $\log t > \log x/(a+1)$, and if $t < y$ then $\log(x/t^a) > \log x/(a+1)$.) The main term integral may be evaluated exactly, and so the main term is
\begin{align}\label{eq:paSn}
\sum_{p\le (x/2^n)^{1/a}} & \frac{1}{p^a}S_n\left(\log_2(x/p^a)\right) \nonumber\\
&=P(a)S_n(\log_2 2^n)-P(a)\big[S_n\left(\log_2(x/t^a)\right) \big]_2^{(x/2^n)^{1/a}} + O_{a,n}(g_n(x)) \nonumber\\
& = P(a)S_n\left(\log_2(x/2^a)\right)+O_{a,n}(g_n(x)) \nonumber\\
& = P(a)S_n\left(\log_2 x\right)+O_{a,n}(g_n(x)) \nonumber\\
& = P(a){\mathcal S}_n(x)+O_{a,n}(g_n(x)),
\end{align}
by \eqref{tenenbaumsresult}.
Here we used $\log_2(x/2^a) = \log_2 x + O_{a}(1/\log x)$ and $\deg(S_n)=n$, so the error term is distributed over less than $n$ factors of $\log_2 x$. Since the implied constant depends only on $a$ and $n$, it can be taken to depend only on the value of $a+n=k$. Plugging back \eqref{eq:paSn} and \eqref{eq:pagn} into \eqref{eq:pam}, we obtain
\begin{align}\label{eq:paml1}
\mathop{\sum_{p^a m\le x}}_{m\in \mathcal{E}} \frac{1}{p^a m} = P(a){\mathcal S}_n(x)+O_{k}(g_n(x)).
\end{align}

This completes the base case $\ell=1$, for any $n\ge0$ and $a\ge2$.

We now turn to the induction step.  Assume that
\begin{displaymath}
\sum_{v=q_1\cdots q_n\prod_{i\le \ell}p_i^{a_i}\le x}\frac{1}{v} = {\mathcal S}_n(x)\prod_{i\le \ell}P(a_i)+O_k(g_n(x)),
\end{displaymath}
where $k=\Omega(v)$.  Let $\mathcal{E}$ be the multi-set of numbers of the form $q_1\cdots q_n\prod_{i\le \ell}p_i^{a_i}$.  Then, for any $a\ge 2$, we have
\begin{displaymath}
\mathop{\sum_{p^a v\le x}}_{v\in \mathcal{E}} \frac{1}{p^a v} = \sum_{p\le (x/2^k)^{1/a}}\frac{1}{p^a}\mathop{\sum_{v\le x/p^a}}_{v\in \mathcal{E}}\frac{1}{v}.
\end{displaymath}
Thus by the induction hypothesis,
\begin{align}\label{eq:paprod}
\mathop{\sum_{p^a v\le x}}_{v\in \mathcal{E}} \frac{1}{p^a v} &= \sum_{p\le (x/2^k)^{1/a}}\frac{1}{p^a}\prod_{i\le \ell}P(a_i)\left({\mathcal S}_n(x/p^a)+O_k\left(g_n(x/p^a)\right)\right) \nonumber\\
&=\prod_{i\le \ell}P(a_i)\cdot\sum_{p\le (x/2^k)^{1/a}} \frac{1}{p^a}S_n\left(\log_2(x/p^a)\right) + O_{k,\ell}\left(\sum_{p\le (x/2^k)^{1/a}}\frac{g_n(x/p^a)}{p^a}\right)
\end{align}
by \eqref{tenenbaumsresult}. As with \eqref{eq:pagn}, the error term above is $\ll_{k,\ell} g_n(x)$. As with \eqref{eq:paSn}, we have
\begin{align*}
\sum_{p\le (x/2^k)^{1/a}}\frac{1}{p^a}S_n\left(\log_2(x/p^a)\right)
& = P(a)S_n\left(\log_2 x\right)+O_{a,k,n}(g_n(x)).
\end{align*}
Finally, we note that the implied constant can be taken to depend only on $a+k$, since for any fixed $a$ and $k$, there are only finitely many possibilities for $\ell$ and $n$.  Therefore,
\begin{displaymath}
\mathop{\sum_{p^a v\le x}}_{v\in \mathcal{E}} \frac{1}{p^a v} = P(a)\prod_{i\le \ell}P(a_i)\cdot {\mathcal S}_n(x)+O_k(g_n(x)).
\end{displaymath}
This completes the proof of Proposition \ref{exponenttype}.
\end{proof}

\section{The Proof of Theorem \ref{mainthm2}}\label{main2}

We now prove Theorem \ref{mainthm2}.  Recall that by Theorem \ref{mainthm}, it suffices to show that $R_k(X+\beta)=V_k(X)$ for each $k$, where $V_k$ is defined as in \eqref{V}.  Recall formulas \eqref{mertensconstant} and \eqref{nu} for the Mertens constant $\beta$ and the function $\nu(z)$. Also recall the prime zeta function $P(s) = \sum_p p^{-s}$.

\begin{lem}\label{taylorexpand}
Define $c_1=c^*_1=\beta$, $c_j = P(j)-(-1)^j\zeta(j)$ and $c^*_j = (-1)^{j+1}(P(j)+\zeta(j))$, $j\ge 2$.  Then for $|z|<1$, we have the expansions
\begin{align}\label{nutaylor}
\nu(z) = \exp\Big(\sum_{j\ge1}\frac{c_j z^j}{j}\Big)\quad \text{and}\quad\nu^*(z) = \exp\Big(\sum_{j\ge1}\frac{c^*_j z^j}{j}\Big).
\end{align}
\end{lem}

\begin{proof}
We prove the expansion for $\nu(z)$.  A similar argument can be used to prove the expansion for $\nu^*(z)$.  The Weierstrass product formula \cite[Theorem II.0.6]{tenenbaumintro} yields a Taylor expansion, for $|z|<1$,
\begin{align*}
\log\Gamma(z+1) = -\gamma z + \sum_{j\ge2}\zeta(j)\frac{(-z)^j}{j}
\end{align*}
and so
\begin{align}\label{WeierstrassGamma}
\frac{1}{\Gamma(z+1)} = \exp\bigg(\gamma z - \sum_{j\ge2} \zeta(j)\frac{(-z)^j}{j}\bigg).
\end{align}

Next, we have
\begin{align*}
\prod_p\Big(1-&\frac{1}{p}\Big)^{z}\Big(1-\frac{z}{p}\Big)^{-1}  = \exp\bigg(\sum_p z\log(1-\tfrac{1}{p})-\log(1-\tfrac{z}{p})\bigg)\\
& = \exp\bigg(z\sum_p\Big( \frac{1}{p} +\log(1-\tfrac{1}{p})\Big) + \sum_p\sum_{j\ge2} \frac{(z/p)^j}{j}\bigg) \\
&= \exp\bigg((\beta-\gamma)z + \sum_{j\ge2} P(j) \frac{z^j}{j}\bigg).
\end{align*}
Combining with \eqref{nu} gives the result.
\end{proof}

From the first assertion of Lemma \ref{taylorexpand}, we may Taylor expand $\nu$ as
\begin{align*}
\nu(z) & = \prod_{j\ge1}\exp\big(c_j z^j/j\big) = \prod_{j\ge1}\sum_{n_j\ge0}\frac{\big(c_j z^j/j\big)^{n_j}}{n_j!} \ = \ \sum_{k\ge0} z^k \sum_{n_1+2n_2+\cdots = k}\prod_{j\ge1}\frac{(c_j/j)^{n_j}}{n_j!}
\end{align*}
from which we see that
\begin{align}\label{eq:nukpartition}
\nu_k := \frac{\nu^{(k)}(0)}{k!} = \sum_{n_1+2n_2+\cdots = k}\prod_{j\ge1}\frac{(c_j/j)^{n_j}}{n_j!}.
\end{align}
Note that $\nu_0 = 1, \nu_1 = \beta$, and recall that $V_k(X) = \sum_{j=0}^k\nu_{k-j}X^j/j!$.

On the other hand, with $R_k$ defined as in Theorem \ref{mainthm}, we have
\begin{align*}
R_k(X+\beta) & = \sum_{n_1+2n_2+\cdots = k}\prod_{j\ge2}\frac{(P(j)/j)^{n_j}}{n_j!}\sum_{i=0}^{n_1}\frac{\lambda_{i,n_1} X^i}{n_1!},\\
&\quad\text{where}\quad \lambda_{i,n} = \frac{n!}{i!}\sum_{m=0}^{n-i}\frac{(\beta-\gamma)^{n-m-i}}{(n-m-i)!} \frac{(1/\Gamma)^{(m)}(1)}{m!}
\end{align*}
is defined as in \eqref{lambda}, so that
\begin{align*}
R_k(X+\beta) & = \sum_{i=0}^k \frac{X^i}{i!}\sum_{n_1+2n_2+\cdots = k}\prod_{j\ge2}\frac{(P(j)/j)^{n_j}}{n_j!}\sum_{m=0}^{n_1-i} \frac{(\beta-\gamma)^{n_1-m-i}}{(n_1-m-i)!} \frac{(1/\Gamma)^{(m)}(1)}{m!}\\
& = \sum_{i=0}^k \frac{X^i}{i!}\sum_{n_1+2n_2+\cdots = k-i}\prod_{j\ge2}\frac{(P(j)/j)^{n_j}}{n_j!}\sum_{m=0}^{n_1} \frac{(\beta-\gamma)^{n_1-m}}{(n_1-m)!} \frac{(1/\Gamma)^{(m)}(1)}{m!} \\
& = \sum_{i=0}^k \frac{X^i}{i!}\sum_{n_1+2n_2+\cdots = k-i}\frac{G^{(n_1)}(0)}{n_1!}\prod_{j\ge2}\frac{(P(j)/j)^{n_j}}{n_j!} =: \sum_{i=0}^k \mu_{k-i}\frac{X^i}{i!} 
\end{align*}
via $n_1\mapsto n_1-i$ (note $n_1\ge i$, otherwise the inner sum on $m$ vanishes), where by the product rule
\begin{align*}
\sum_{m=0}^{n_1}\frac{(\beta-\gamma)^{n_1-m}}{(n_1-m)!} \frac{(1/\Gamma)^{(m)}(1)}{m!} & = \frac{1}{n_1!}\frac{d^{n_1}}{dz^{n_1}}\Big[\frac{e^{(\beta-\gamma) z}}{\Gamma(z+1)}\Big]_{z=0} = \frac{G^{(n_1)}(0)}{n_1!},
\end{align*}
for $G(z) = e^{(\beta-\gamma) z}/\Gamma(z+1)$. Also by the product rule,
\begin{align*}
\mu_k : & = \sum_{n_1=0}^k\frac{G^{(n_1)}(0)}{n_1!}\sum_{2n_2+\cdots = k-n_1}\prod_{j\ge2}\frac{(P(j)/j)^{n_j}}{n_j!}\\ & = \sum_{n_1+m=k}\frac{G^{(n_1)}(0)}{n_1!}\frac{1}{m!}\frac{d^m}{dz^m}\Big[\exp\Big(\sum_{j\ge2}P(j)\frac{z^j}{j}\Big) \Big]_{z=0}\\
& = \frac{1}{k!}\frac{d^k}{dz^k}\Big[G(z)\exp\Big(\sum_{j\ge2}P(j)\frac{z^j}{j}\Big) \Big]_{z=0} = \frac{\nu^{(k)}(0)}{k!} = \nu_k,
\end{align*}
recalling the expansion \eqref{nutaylor}.  This completes the proof of Theorem \ref{mainthm2}.

\section{Proof of Equivalence of Theorems \ref{tenenbaumthm} and \ref{qihuthm}}\label{qihutenenb}

In this section we provide a direct proof that the coefficients appearing in Tenenbaum's formula \eqref{tenenbaumsresult} and the formula \eqref{qihuresult} are equal.  First, we let $b_m = a_m/m!$ so that
\begin{align*}
b_m = \frac{1}{m}\sum_{i=1}^{m-1}(-1)^i \zeta(i+1)b_{m-1-i} = \frac{1}{m}\sum_{j=2}^{m}(-1)^{j-1} \zeta(j)b_{m-j}.
\end{align*}
This recursive identity for $b_m$ implies that, by \cite[Lemma 2.2]{lichtman2}, $b_m$ is given explicitly as
\begin{align*}
b_m = \sum_{2n_2 + 3n_3\cdots = m}\prod_{j\ge2}\frac{((-1)^{j-1} \zeta(j)/j)^{n_j}}{n_j!} = \frac{1}{m!}\frac{\dd^m}{\dd z^m}\Big[\exp\Big(-\sum_{j\ge2} \zeta(j)(-z)^j/j\Big)\Big]_{z=0}
\end{align*}
and so the Weierstrass product formula in \eqref{WeierstrassGamma} gives
\begin{align}
a_m = m!\,b_m &= \frac{\dd^m}{\dd z^m}\Big[\frac{e^{-\gamma z}}{\Gamma(z+1)}\Big]_{z=0} = \sum_{i=0}^m \binom{m}{i} (-\gamma)^{m-i}\Big(\frac{1}{\Gamma}\Big)^{(i)}(1).
\end{align}
Thus by the binomial theorem, the main term in \eqref{qihuresult} is
\begin{align*}
{\mathcal S}_k(x)+{\rm err.}& =\sum_{m=0}^k \binom{k}{m} a_m(\log_2 x+\beta)^{k-m}\\
& = \sum_{m=0}^k \binom{k}{m} a_m \sum_{j=0}^{k-m}\binom{k-m}{j}\beta^{k-m-j}(\log_2 x)^j\\
& = \sum_{m=0}^k \binom{k}{m} \sum_{i=0}^m \binom{m}{i} (-\gamma)^{m-i}\Big(\frac{1}{\Gamma}\Big)^{(i)}(1) \sum_{j=0}^{k-m}\binom{k-m}{j}\beta^{k-m-j}(\log_2 x)^j\\
& =  \sum_{j=0}^k (\log_2 x)^j \sum_{i=0}^{k-j}\Big(\frac{1}{\Gamma}\Big)^{(i)}(1)\sum_{m=i}^{k-j}
\frac{k!}{i!(m-i)!j!(k-m-j)!} (-\gamma)^{m-i}\beta^{k-m-j}\\
& =  \sum_{j=0}^k (\log_2 x)^j \sum_{i=0}^{k-j}\Big(\frac{1}{\Gamma}\Big)^{(i)}(1)\frac{k!}{i!j!(k-j-i)!}\sum_{m=0}^{k-j-i} \binom{k-i-j}{m} (-\gamma)^m\beta^{k-m-i-j}
\end{align*}
Hence we obtain Tenenbaum's main term in \eqref{tenenbaumsresult},
\begin{align*}
{\mathcal S}_k(x)+{\rm err.}
& =  \sum_{j=0}^k (\log_2 x)^j\sum_{i=0}^{k-j}\binom{k}{i,j,k-j-i}(\beta-\gamma)^{k-j-i}\Big(\frac{1}{\Gamma}\Big)^{(i)}(1) 
\ = \ S_k(\log_2 x).
\end{align*}
This completes the proof.

\section{The Proof of Theorem \ref{nuexpand}}\label{recurrence}

We first establish the following recurrence relations for the sequences $(\nu_k)_{k=0}^\infty$ and $(\nu^*_k)_{k=0}^\infty$. 
These recurrences are analogous to those of $(d_k)_{k=0}^\infty$ in the smooth setting of Theorem \ref{thm:smooth} above, as well as Proposition 3.1 from \cite{lichtman}.

\begin{prop}\label{prop:recurrence}
Define $c_1=c^*_1=\beta$, $c_j = P(j)-(-1)^j\zeta(j)$ and $c^*_j = (-1)^{j+1}(P(j)+\zeta(j))$, $j\ge 2$.  Then the sequences $(\nu_k)_{k=0}^\infty$ and $(\nu^*_k)_{k=0}^\infty$ are given recursively by $\nu_0=\nu^*_0=1$,
\begin{displaymath}
\nu_k = \frac{1}{k}\sum_{j=1}^k \nu_{k-j} c_j\quad \text{and}\quad \nu^*_k = \frac{1}{k}\sum_{j=1}^k \nu^*_{k-j} c^*_j.
\end{displaymath}
\end{prop}
\begin{proof}
We prove the relation for $\nu_k$ and note that a similar argument gives that of $\nu^*_k$.  We proceed by induction on $k\ge1$. For the base case $k=1$, we have $\nu_1 = \beta = \nu_0c_1$. Now assume the claim for all $1\le r<k$. By the explicit formula \eqref{eq:nukpartition} for $\nu_k$ we have
\begin{align*}
\sum_{r=1}^k \nu_{k-r}c_r & = \sum_{r=1}^k c_r\,\sum_{n_1 +\cdots = k-r}\prod_{j\ge1} \frac{(c_j/j)^{n_j}}{n_j!} = \sum_{r=1}^k \sum_{n_1 +\cdots = k-r} \frac{c_r^{n_r+1}}{r^{n_r}\,n_r!}\prod_{\substack{j\ge1\\ j\neq r}} \frac{(c_j/j)^{n_j}}{n_j!}\\
& = \sum_{r=1}^k \sum_{\substack{n_1 +\cdots = k\\ n_r\ge1}}rn_r\,\prod_{j\ge1} \frac{(c_j/j)^{n_j}}{n_j!} = \sum_{n_1 +\cdots = k}\prod_{j\ge1} \frac{(c_j/j)^{n_j}}{n_j!}\sum_{\substack{1\le r\le k\\n_r\ge1}}rn_r = k\nu_k.
\end{align*}
In the last step, we dropped the condition $n_r\ge 1$ (since $rn_r=0$ for $n_r=0$) which gives $\sum_{r=1}^k rn_r=k$. This completes the proof.
\end{proof}

Note this recursion enables rapid computation of $\nu_k,\nu^*_k$ to high precision. We show results for $k\le 10$ below.
\[\begin{array}{rrr}
k & \nu_k & \nu^*_k\\
\hline
0 & 1 & 1       \\
1 & \beta=2.61497\cdot10^{-1} & \beta=2.61497\cdot10^{-1}\\
2 & -5.62153\cdot10^{-1} & -1.01440\cdot 10^0\\
3 & 3.05978\cdot10^{-1} & 1.87717\cdot 10^{-1}\\
4 & 2.62973\cdot10^{-2}  & 3.44297\cdot 10^{-1}\\
5 & -6.44501\cdot10^{-2} & -1.86153\cdot 10^{-1}\\
6 & 3.64064\cdot10^{-2}   & -1.50297\cdot 10^{-2}\\
7 & -4.70865\cdot10^{-3} & 4.29836\cdot 10^{-2}\\
8 & -4.33984\cdot10^{-4}     & -1.30388\cdot 10^{-2}\\
9 & 1.5085\cdot 10^{-3} & -1.57532\cdot 10^{-3}\\
10 & -1.83548\cdot10^{-4}    & 2.17630\cdot 10^{-3}\\    
\end{array}\]

Now we prove the main theorem of the section.

\begin{proof}[Proof of Theorem \ref{nuexpand}]
Note that $\nu(z) = G(z)C(z)$, where
\begin{align*}
G(z) & := \frac{1}{\Gamma(z+1)}\prod_p \Big(1-\frac{1}{p}\Big)^z e^{z/p} = \frac{e^{(\beta-\gamma)z}}{\Gamma(z+1)},\\
C(z) & := \prod_{p}\Big(1-\frac{z}{p}\Big)^{-1} e^{-z/p} = \exp(\sum_{j\ge2}P(j)\frac{z^j}{j}).
\end{align*}
By the product rule, we thus have
\begin{align}\label{eq:findvk}
\nu_k = \sum_{n=0}^k\frac{G^{(n)}(0)}{n!} d_{k-n},
\end{align}
where
\begin{align*}
d_k & := \frac{1}{k!}C^{(k)}(0) = \sum_{2n_2+3n_3+\cdots=k}\prod_{j\ge2}\frac{(P(j)/j)^{n_j}}{n_j!}.
\end{align*}
Now we appeal to Theorem \ref{thm:smooth}, which gives
\begin{align*}
d_k = \sum_{p<q}\delta_p \,p^{-k} \ + \ O_q(q^{-k})
\end{align*}
for any prime $q$, where 
\begin{align}\label{eq:findap}
\delta_p & := e^{-1}\prod_{p'\neq p}\Big(1-\frac{p}{p'}\Big)^{-1}e^{-p/p'} = \lim_{z\to p}(1-z/p) C(z).
\end{align}
Hence \eqref{eq:findvk} becomes
\begin{align}\label{nukexpansion}
\nu_k = \sum_{p<q}\delta_p\sum_{n=0}^k\frac{G^{(n)}(0)}{n!}\,p^{n-k} \ + \ O_q\Big(\sum_{n=0}^k\frac{G^{(n)}(0)}{n!}q^{n-k}\Big).
\end{align}

Since $G$ is entire, we have by estimate \eqref{nukexpansion} that
\begin{displaymath}
\begin{aligned}
\nu_k &= \sum_{p<q}\delta_p p^{-k}\left(G(p)-\sum_{n>k}\frac{G^{(n)}(0)}{n!}p^n\right)+O_q(G(q)q^{-k})\\
&= \sum_{p<q}\beta_p p^{-k} \ + \ O_q(q^{-k}),
\end{aligned}
\end{displaymath}
where the last equality holds by Lemmas \ref{deltabeta} and \ref{lem:Gk0bound} below.\qedhere
\end{proof}

\begin{lem}\label{deltabeta}
We have $\delta_p G(p) = \beta_p$ for each prime $p$.
\end{lem}

\begin{proof}
By \eqref{betap}, \eqref{eq:findap}, and the definition of $G$, we have
\begin{align*}
\frac{\delta_p G(p)}{\beta_p} &= e^{(\beta-\gamma)p-1}\left(1-\frac{1}{p}\right)^{-p}\prod_{p'\ne p}\left(1-\frac{1}{p'}\right)^{-p}e^{-p/p'} = e^{(\beta-\gamma)p}\prod_{p'}\left(1-\frac{1}{p'}\right)^{-p}e^{-p/p'}\\
&= e^{(\beta-\gamma)p}\exp\left(-p\sum_{p'}\left(\log\left(1-\frac{1}{p'}\right)+\frac{1}{p'}\right)\right) =1.
\end{align*}
\end{proof}

\begin{lem}\label{lem:Gk0bound}
For any $m\ge 2$, we have $G^{(k)}(0)/k! \ll_m m^{-k}$ as $k\to\infty$.
\end{lem}
\begin{proof}
It suffices to prove $G^{(k)}(0)/k! \ll_m (k+1)^m m^{-k}$ for any $m\ge 2$, in which case taking $m'=m+1$ gives the result. Now to show this, given a fixed $m$ we have
\begin{align}
G(z) & = \exp(\beta z - \sum_{j\ge2}\zeta(j)(-z)^j/j) \nonumber\\
& = e^{\beta z}\exp(\sum_{1\le n<m}[\log(1+z/n)-z/n] - \sum_{j\ge 2} \bar\zeta(j)(-z)^j/j) = G_0(z) G_1(z)
\end{align}
letting $G_0(z) = e^{\beta z}\prod_{n< m}(1+z/n)e^{-z/n}$ and $G_1(z) = \exp(-\sum_{j\ge2}\bar\zeta(j)(-z)^j/j)$, where $\bar\zeta(j) := \sum_{n\ge m}n^{-j}$.  Note that $\bar\zeta(2)=\sum_{n\ge m}n^{-2}\le m^{-2}+\int_m^\infty t^{-2}\dd{t} = (m+1)/m^2$, so by induction $\bar\zeta(j) \le (m+1)/m^j$ for all $j\ge2$. Thus
\begin{align*}
G_1^{(k)}(0) &= (-1)^k k!\sum_{2n_2+\cdots=k}\prod_{j\ge 2} \frac{(-\bar\zeta(j)/j)^{n_j}}{n_j!},\\
|G_1^{(k)}(0)|& \le  k!\sum_{2n_2+\cdots=k}\prod_{j\ge 2}\frac{((m+1)/j m^j)^{n_j}}{n_j!} = \widetilde G_1^{(k)}(0)
\end{align*}
for
\begin{align*}
\widetilde G_1(z) & := \exp((m+1)\sum_{j\ge 2}(z/m)^j/j)
= \frac{e^{-(m+1)(z/m)}}{(1-z/m)^{m+1}}.
\end{align*}
Note the derivatives
\begin{align*}
\frac{1}{k!}\frac{\dd^k}{\dd z^k}\Big[\frac{1}{(1-z/m)^{m+1}}\Big]_{z=0} &= \frac{m^{-k}}{m!}\prod_{1\le j\le m}(k+j) \le \frac{(k+1)^m}{m^k},\\
\frac{1}{k!}\frac{\dd^k}{\dd z^k}\big[e^{-(m+1)(z/m)}\big]_{z=0}  &= \frac{1}{k!}\big(-(m+1)/m\big)^k \le \frac{(1+1/m)^k}{k!}
\end{align*}
and since $\sum_{k\in K}u^k/k! \le e^{|u|}$ for any $u\in\mathbb{R}$, any set $K\subset \mathbb{N}$, by the product rule we have
\begin{align*}
\frac{\widetilde G_1^{(k)}(0)}{k!} \le \sum_{a+b=k}\frac{(a+1)^m}{m^a}\frac{(1+1/m)^b}{b!} &= m^{-k}\sum_{b=0}^k (k+1-b)^m\frac{(m+1)^b}{b!}\\
&\le e^{m+1} \frac{(k+1)^m}{m^k} \ \ll_m \  \frac{(k+1)^m}{m^k}.
\end{align*}
We also have
\begin{align*}
G_0(z) & = e^{\beta z}\prod_{j< m}(1+z/j)e^{-z/j} = \sum_{k\ge0} \frac{z^k}{k!}(\beta-\sum_{j<m}1/j)^k\prod_{j< m}(1+z/j).
\end{align*}
In particular $G_0^{(k)}(0) \ll_m |\beta-\sum_{j<m}1/j|^k$. Thus by the product rule for $G=G_0\cdot G_1$ we have
\begin{align*}
\frac{G^{(k)}(0)}{k!} = \sum_{b+c=k}\frac{G_0^{(b)}(0)}{b!}\frac{G_1^{(c)}(0)}{c!},\\
\frac{|G^{(k)}(0)|}{k!} \le \sum_{b+c=k}\frac{|G_0^{(b)}(0)|}{b!}\,\frac{\widetilde G_1^{(c)}(0)}{c!} &\ll_m \sum_{b+c=k} \frac{|\beta-\sum_{j<m}1/j|^b}{b!}\, \frac{(c+1)^m}{m^c}\\
&\ll_m \frac{(k+1)^m}{m^k}.
\end{align*}
This gives the claim as desired.
\end{proof}

\section{The Proofs of Theorems \ref{r2theorem} and \ref{explicitr2theorem}}\label{R2}

Recall that $E(x):= \sum_{p\le x} 1/p - (\log_2 x + \beta)$. Note $E(t)\ll_A (\log t)^{-A}$ for all $A> 0$, which implies $\int_2^\infty |E(t)| (\log t)^j /t \dd{t}$ converges for all $j$. So we may define the constants
\begin{align}\label{eq:defgammaj}
\gamma_j = \int_2^\infty E(t)(\log t)^{j-1}\frac{\dd{t}}{t}.
\end{align}
In particular, we have the relation $\alpha_j = (\log 2)^j(1/j-\log_22 -\beta )/j+\gamma_j$.

We cite some useful lemmas.

\begin{lem}[Rosser \& Schoenfeld {\cite[Theorems\ 5 \& 20]{rosser1962approximate}}]\label{RS}
We have
\begin{equation}\label{RS1}
    -1/(2\log^2 x) < E(x) < 1/\log^2 x,~~(x>1),
\end{equation}
\begin{equation}\label{RS3}
    0<E(x),~~(1<x\le 10^8).
\end{equation}
\end{lem}

\begin{lem}[Dusart {\cite[Theorem 6.10]{dusart2010},\cite[Theorem 5.6]{dusart2016}}]\label{D} 
We have
\begin{equation}\label{D1}
    |E(x)|\le1/(10\log^2 x)+4/(15\log^3 x),~~(x\ge 10372),
\end{equation}
\begin{equation}\label{D2}
    |E(x)|\le 0.2/\log^3 x,~~(x\ge 2278383).\footnote{We were recently made aware of an issue in \cite{dusart2016} which affects \eqref{D2}.  This has been resolved in a paper \cite{broadbent} of Broadbent et al.\ which provides the necessary bounds on Chebyshev's function $\theta(x)$ to guarantee the validity of \eqref{D2}.}
\end{equation}
\end{lem}


For ease of notation, we define 
\begin{align*}
T(x):= {\mathcal R}_1(x)=\sum_{p\le x}\frac{1}{p}.
\end{align*}

\begin{lem}\label{integral}
For all $N\ge 0$, and $\gamma_j$ as in \eqref{eq:defgammaj}, we have
\begin{displaymath}
\int_2^{\sqrt x}\frac{E(t)}{t\log(x/t)}\dd{t} = \sum_{j=1}^N\frac{\gamma_j}{\log^j x} + O_N\left((\log x)^{-N-1}\right).
\end{displaymath}
\end{lem}

\begin{proof}
Note for $t\le \sqrt{x}$, we have the geometric series 
\begin{align*}
\sum_{j=0}^{N-1}\Big(\frac{\log t}{\log x}\Big)^j = \frac{1-(\log t/\log x)^N}{1-(\log t/\log x)} = \frac{\log x}{\log(x/t)}(1-O(2^{-N})).
\end{align*}
So recalling $E(t)\ll_A (\log t)^{-A}$ for all $A>0$, we may interchange sum and integral to obtain
\begin{align*}
\int_2^{\sqrt x}\frac{E(t)}{t\log(x/t)}\dd{t} = \sum_{j=0}^{N-1}\frac{1+O(2^{-N})}{(\log x)^{j+1}}\int_2^{\sqrt x} (\log t)^j E(t)\frac{\dd{t}}{t}
 \ = \ \sum_{j=1}^N \frac{\gamma_j}{(\log x)^j} + O_N((\log x)^{-N-1})
\end{align*}
since for all $j\le N$, letting $A=2(N+1)$,\footnote{This corrects the published version, where it is stated that $A=2N$.}
\begin{align*}
\int_{\sqrt x}^\infty (\log t)^j E(t)\frac{\dd{t}}{t} \ \ll_N \ (\log x)^{-N-1}.
\end{align*}
This completes the proof.
\end{proof}

We give the following limit characterization of the constants $\alpha_j$.

\begin{lem}\label{lem:alphaj}
For all $N\ge j\ge 1$, and $\alpha_j$ as in \eqref{eq:alphaj}, we have
\begin{align*}
\alpha_j = \frac{1}{j}\left(\frac{\log^j x}{j} - \sum_{p\le x}\frac{\log^j p}{p}\right) + O_N((\log x)^{-N}).
\end{align*}
\end{lem}
\begin{proof}
By partial summation and Mertens' second theorem, we have
\begin{displaymath}
\begin{aligned}
\sum_{p\le x}\frac{\log^j p}{p} &= T (x)\log ^j x -j\int_2^x T(t)\log^{j-1} t \frac{\dd{t}}{t}\\
&= (\log_2 x +\beta + E (x))\log^j x - j\int_2^x (\log_2 t +\beta + E(t)) \log^{j-1} t\frac{\dd{t}}{t}\\
&= (\log_2 x+E(x))\log^j x - \left[u^j(\log u-1/j)\right]_{\log 2}^{\log x}+ \beta \log^j 2 -j\int_2^x E(t) \log^{j-1} t \frac{\dd{t}}{t}\\
& = E(x)\log^j x + \log^j x/j +\log^j 2(\log_2 2+\beta-1/j)  -j\int_2^x E(t) \log^{j-1} t \frac{\dd{t}}{t}.
\end{aligned}
\end{displaymath}
Hence by definition of $\alpha_j$,
\begin{align*}
\frac{1}{j}\left(\frac{\log^j x}{j} - \sum_{p\le x}\frac{\log^j p}{p}\right) = \alpha_j -E(x)\log^j x/j - \int_x^\infty E(t) \log^{j-1} t \frac{\dd{t}}{t}.
\end{align*}
Recalling $E(x)\ll_N (\log x)^{-2N}$ completes the proof.
\end{proof}

\begin{lem}\label{R2lemma}
We have
\begin{displaymath}
{\mathcal R}_2(x) = \sum_{p\le \sqrt{x}}\frac{1}{p}T\left(x/p\right) - \frac{1}{2}T(\sqrt{x})^2 + \frac{1}{2}\sum_{p\le \sqrt{x}}\frac{1}{p^2}.
\end{displaymath}
\end{lem}

\begin{proof}
Note that $\mathbb{N}_2\cap [1,x] = \{pq\le x: p < q\}\cup \{p^2: p\le \sqrt{x}\}$.  Also, we have $pq\le x, p<q$ if and only if $p\le \sqrt{x}$ and $p<q\le x/p$.  Therefore,
\begin{displaymath}
\begin{aligned}
{\mathcal R}_2(x) &= \sum_{p\le \sqrt{x}}\sum_{p<q\le \frac{x}{p}}\frac{1}{pq} + \sum_{p\le \sqrt{x}}\frac{1}{p^2}\\
& = \sum_{p\le \sqrt x}\frac{1}{p}\left(T\left(x/p\right)-T(p)\right) + \sum_{p\le \sqrt{x}}\frac{1}{p^2}.
\end{aligned}
\end{displaymath}
Thus it suffices to note that by the multinomial theorem, we have
\begin{displaymath}
2\sum_{p\le \sqrt x}\frac{T(p)}{p} = \sum_{p\le \sqrt x}\frac{1}{p^2} + \left(\sum_{p\le \sqrt x}\frac{1}{p}\right)^2.\qedhere
\end{displaymath}
\end{proof}

\begin{proof}[Proof of Theorem \ref{r2theorem}]
By Lemma \ref{R2lemma}, ${\mathcal R}_2(x)=A_1(x)+A_2(x)+A_3(x)$, where
\begin{displaymath}
\begin{aligned}
& A_1(x):=\sum_{p\le \sqrt x}\frac{1}{p}T\left(x/p\right)=\sum_{p\le \sqrt x}\frac{\log_2(x/p)+\beta}{p}+\sum_{p\le \sqrt x}\frac{E\left(x/p\right)}{p},\\
& A_2(x) := -\frac{1}{2}T(\sqrt x)^2 = -\frac{1}{2}(\log_2\sqrt x + \beta + E(\sqrt x))^2,\\
& A_3(x) := \frac{1}{2}\sum_{p\le \sqrt x}p^{-2}.
\end{aligned}
\end{displaymath}
Similarly, ${\mathcal R}_2^*(x)$ is given by negating the last term above.
We next write $A_1(x)=B_1(x)+B_2(x)+B_3(x)$, where
\begin{displaymath}
\begin{aligned}
B_1(x):= \sum_{p\le \sqrt x}\frac{\log_2(x/p)}{p},\quad
B_2(x) := \sum_{p\le \sqrt x}\frac{\beta}{p},\quad
B_3(x) := \sum_{p\le \sqrt x}\frac{E(x/p)}{p}.
\end{aligned}
\end{displaymath}
We have $B_2(x) = \beta\cdot T(\sqrt x) = \beta(\log_2\sqrt x+\beta+E(\sqrt x))$.  By partial summation,
\begin{displaymath}
\begin{aligned}
B_1(x) &= T(\sqrt x)\log_2\sqrt x + \int_2^{\sqrt x}\frac{T(t)}{t\log(x/t)}\dd{t}\\
& = (\log_2\sqrt x + \beta+ E(\sqrt x))\log_2\sqrt x  \ + \int_{\log 2}^{\log\sqrt x}\frac{\log u+\beta}{\log x - u}\dd{u} + \int_2^{\sqrt x}\frac{E(t)}{t\log(x/t)}\dd{t}.
\end{aligned}
\end{displaymath}
Denoting the dilogarithm $\text{Li}_2(z) = \sum_{j\ge1}z^j/j^2$, the integral on the left equals
\begin{align*}
\int_{\log 2}^{\log\sqrt x}\frac{\log u+\beta}{\log x - u}\dd{u} 
&= -\bigg[(\beta+\log u)\log(1-\tfrac{u}{\log x}) + \text{Li}_2(\tfrac{u}{\log x})\bigg]_{\log 2}^{\frac{1}{2}\log x}\\
&= (\beta+\log_2 2)\log(1-\tfrac{\log 2}{\log x}) + \text{Li}_2(\tfrac{\log 2}{\log x})\\
&\quad + (\beta+\log_2 \sqrt{x})\log 2 - \text{Li}_2(\tfrac{1}{2}).
\end{align*}
After simplifying
and noting $\text{Li}_2(1/2) = \zeta(2)/2-(\log2)^2/2$, we have
\begin{align}\label{R2formula}
{\mathcal R}_2(x) &= \frac{1}{2}(\log_2 x+\beta)^2 + A_3(x)  - \frac{\zeta(2)}{2}  + (\beta+\log_2 2)\log\left(1-\tfrac{\log 2}{\log x}\right)+ \text{Li}_2\left(\tfrac{\log 2}{\log x}\right) \nonumber\\
& \quad -\frac{1}{2}E(\sqrt x)^2 +
\sum_{p\le \sqrt x}\frac{E\left(x/p\right)}{p}+\int_2^{\sqrt x} \frac{E(t)}{t\log(x/t)}\dd{t}.
\end{align}
By Lemma \ref{pn}, $A_3(x) =  P(2)/2 + O(x^{-1/2})$. And $E(t) \ll_A 1/\log^A t$ implies
\begin{displaymath}
E(\sqrt x)^2 \ \text{and} \; \sum_{p\le \sqrt x}\frac{E\left(x/p\right)}{p} \ \ll_N \ \frac{1}{\log^{N+1}x}
\end{displaymath}
using Mertens' second theorem. Thus by Lemma \ref{integral},
\begin{align*}
{\mathcal R}_2(x) &=\frac{1}{2}(\log_2 x+\beta)^2+\frac{P(2)-\zeta(2)}{2}+\sum_{j=1}^N\frac{\gamma_j}{\log^j x} \nonumber\\
& \quad + (\beta+\log_2 2)\log\left(1-\tfrac{\log 2}{\log x}\right) + \text{Li}_2\left(\tfrac{\log 2}{\log x}\right) \ + \ O_N\left(\frac{1}{\log^{N+1}x}\right)\\
&=\frac{1}{2}(\log_2 x+\beta)^2+\frac{P(2)-\zeta(2)}{2}+\sum_{j=1}^N\frac{\gamma_j}{\log^j x}\\
& \quad + \sum_{j=1}^N \frac{\log^j 2}{j\log^j x}\Big(\frac{1}{j}-\log_22-\beta\Big) \ + \ O_N\left(\frac{1}{\log^{N+1}x}\right)
\end{align*}
using $\text{Li}_2(z) = \sum_{j\ge1}z^j/j^2$ and $\log(1-z)=-\sum_{j\ge1} z^j/j$ with $z=\log 2/\log x$. 

Hence recalling $\alpha_j=(\log 2)^j(1/j-\log_2-\beta)/j+\gamma_j$ completes the proof.
\end{proof}

\begin{proof}[Proof of Theorem \ref{explicitr2theorem}]
We verify the bound for all $x\le x_0 := 10^9$ by computer.  Let $x > x_0$.  We bound the terms in \eqref{R2formula}.  By \eqref{D1} we have $-0.0003/\log^2 x < -E(\sqrt x)^2\le 0.$    By Lemma \ref{integral} we have
\begin{displaymath}
\begin{aligned}
\int_2^{\sqrt x}\frac{E(t)}{t\log(x/t)}\dd{t} &= \frac{\gamma_1}{\log x} - \frac{1}{\log x}\int_{\sqrt x}^\infty\frac{E(t)}{t}\dd{t} + \frac{1}{\log x}\int_2^{\sqrt x}\frac{E(t)}{t}\frac{\log t}{\log(x/t)}\dd{t}\\
&=\gamma_1/\log x + I_1+I_2,
\end{aligned}
\end{displaymath}
say.  To bound $I_1$ above (resp.\ below) we use \eqref{RS3} and \eqref{D2} (resp.\ \eqref{D1}), obtaining $-0.2515/\log^2 x < I_1 < 0.0194/\log^2 x$.  To bound $I_2$ above (resp.\ below) we use \eqref{RS1} (resp.\ \eqref{RS3} and \eqref{D2}), obtaining
\begin{displaymath}
-\frac{0.0218}{\log^2 x} < I_2 < \frac{\log_2(x/2)-\log_2 2}{\log^2 x}.
\end{displaymath}
By Lemma \ref{D} and following the method in \cite[Theorem 5.1]{BKK}, we have
\begin{displaymath}
-\frac{0.2161}{\log^2 x} < \sum_{p\le \sqrt x}\frac{E\left(x/p\right)}{p} < \sum_{p\le \sqrt x}\frac{0.1258}{p\log^2(x/p)}\le \frac{0.1258\log_2 x+0.1593}{\log^2 x}.
\end{displaymath}
Next, by Lemma \ref{pn} we have
\begin{displaymath}
-\frac{0.0007}{\log^2 x} < -\frac{1}{2} \sum_{p>\sqrt x}p^{-2} < 0.
\end{displaymath}
We find by bounding series expansions that
\begin{displaymath}
-\frac{0.2458\beta}{\log^2 x} < \beta\log\left(1-\frac{\log 2}{\log x}\right) + \frac{\beta\log 2}{\log x}< -\frac{0.2402\beta}{\log^2 x},
\end{displaymath}
\begin{displaymath}
\frac{0.1201}{\log^2 x} < \text{Li}_2\left(\frac{\log 2}{\log x}\right) - \frac{\log 2}{\log x} < \frac{0.1221}{\log^2 x},
\end{displaymath}
and
\begin{displaymath}
\frac{0.0818}{\log^2 x} < (\log_2 2)\log\left(1-\frac{\log 2}{\log x}\right) + \frac{\log_2 2 \log 2}{\log x} < \frac{0.0962}{\log^2 x},
\end{displaymath}
where we note that $-\log_2 2>0$ for the last inequality above.  Combining all bounds, we complete the proof of Theorem \ref{explicitr2theorem}.
\end{proof}

\section{A new proof of Theorem \ref{N2}}\label{sec:N2}

We provide an alternate proof of Crisan and Erban's asymptotic expansion for $\mathcal{N}_2(x)$ using our refined estimate for $\mathcal{R}_2(x)$ given in Theorem \ref{r2theorem}. The proof relies on a strong form of the prime number theorem, i.e. $E(x)\ll_A (\log x)^{-A}$ for all $A> 0$.

\begin{proof}[Proof of Theorem \ref{N2}]
By partial summation,
\begin{align*}
\mathcal{N}_2(x) &=x\mathcal{R}_2(x) - \int_4^x \mathcal{R}_2(t) \dd{t}.
\end{align*}
So by Theorem \ref{r2theorem}, at admissible error $\epsilon(x) \ll x\log_2 x/\log^{N+1} x$, we have
\begin{align}\label{eq:N2error}
\mathcal{N}_2(x)& + \epsilon(x) \\
&=\frac{x}{2}(\log_2 x+\beta)^2 + x\nu_2 + \sum_{j=1}^N\frac{x\alpha_j}{\log^j x}
- \int_4^x \bigg(\frac{1}{2}(\log_2 t+\beta)^2 + \nu_2 + \sum_{j=1}^N\frac{\alpha_j}{\log^j t}\bigg) \dd{t} \nonumber\\
&= \frac{x}{2}(\log_2 x)^2+\beta x\log_2 x -\int_4^x \bigg(\frac{1}{2}(\log_2 t)^2+\beta \log_2 t\bigg) \dd{t} + \sum_{j=1}^N \alpha_j\Big(\frac{x}{\log^j x} - \int_4^x \frac{\dd{t}}{\log^j t}\Big) \nonumber\\
&= (\log_2 x+\beta)\text{li}(x)-\int_4^x \frac{\text{li}(t)}{t\log t} \dd{t}-\sum_{j=1}^{N-1} \frac{j\alpha_j}{j!}\sum_{n=j}^{N-1}\frac{n!\,x}{\log^{n+1} x}\nonumber
\end{align}
using the expansions
\begin{align*}
\int_4^x \log_2 t\dd{t} &= x\log_2 x - \text{li}(x) + O(1),\\
\int_4^x (\log_2 t)^2 \dd{t} &= x(\log_2 x)^2 - 2\text{li}(x)\log_2 x +O(1) + 2\int_4^x \frac{\text{li}(t)}{t\log t}\dd{t},\\
\int_4^x \frac{\dd{t}}{\log^{j+1} t} &= \sum_{n=j}^{N-1} \frac{n!}{j!}\frac{x}{\log^{n+1} x} + O_N\left(\frac{x}{\log^{N+1} x}\right) \qquad (j\ge0).
\end{align*}
Note $j=0$ gives expansion for the logarithmic integral $\text{li}(x)=\int_2^x \dd{t}/\log t$, so that \eqref{eq:N2error} becomes
\begin{align*}
&(\log_2 x +\beta)\sum_{n=0}^{N-1} \frac{n!\,x}{\log^{n+1} x}- \sum_{n=0}^{N-1} n!\int_4^x\frac{\dd{t}}{\log^{n+2} t} -\sum_{j=1}^{N-1} \frac{j\alpha_j}{j!}\sum_{n=j}^{N-1}\frac{n!\,x}{\log^{n+1} x}\\
&=(\log_2 x +\beta)\sum_{n=0}^{N-1} \frac{n!\,x}{\log^{n+1} x}- \sum_{n=0}^{N-1}\frac{1}{n+1}\sum_{m=n+1}^{N-1} \frac{m!\,x}{\log^{m+1} x} -\sum_{j=1}^{N-1} \frac{j\alpha_j}{j!}\sum_{n=j}^{N-1}\frac{n!\,x}{\log^{n+1} x}.
\end{align*}
Hence we conclude
\begin{align*}
\mathcal{N}_2(x) + \epsilon(x) = \  \sum_{n=0}^{N-1} \frac{n!\,x}{\log^{n+1} x}\Big(\log_2 x+\beta - \sum_{j=1}^n\Big(\frac{j \alpha_j}{j!}+\frac{1}{j}\Big)\Big).
\end{align*}
By definition of $D_n$, this completes the proof of Theorem \ref{N2}.
\end{proof}

\section*{Acknowledgments}
We are grateful to Carl Pomerance for valuable discussions with regard to both mathematics and exposition. We thank Scott Lambert for helpful comments involving Theorem \ref{r2theorem}, and Ofir Gorodetsky for communicating a sketch of \eqref{eq:alphaj1j}. We also acknowledge the anonymous referee, as well as Nathan Ng for informing us of an error in \cite{dusart2016}. The third author is supported by a Clarendon Scholarship at the University of Oxford.

\end{document}